\documentclass[a4paper,12pt,english]{article}
\usepackage[cp1251]{inputenc}
\usepackage[T2A]{fontenc}
\usepackage[english]{babel}
\usepackage[dvips]{graphicx}

\usepackage{amsmath}
\usepackage{amssymb}
\usepackage{amsxtra}
\usepackage{amsthm}
\usepackage{latexsym}

\usepackage{geometry}
\def\mbf{\mathbf{f}}
\def\mbA{\mathbf{A}}
\def\bo{{\boldsymbol \omega}}
\def\bde{{\boldsymbol \varepsilon}}
\def\grad{\mathop{\rm grad }\nolimits}
\def\oq{\omega}
\def\Oq{\Omega}
\def\aq{\alpha}
\def\tq{\theta}
\def\Lqq{\Lambda}
\def\rmd{{\rm d}}
\def\rmi{{\rm i}}
\def\vk{\varkappa}
\def\sg{\sigma}
\def\prt{\partial}
\def\ds{\displaystyle}
\def\bR{\mathbf{R}}

\newcommand{\bb}[1]{\langle #1 \rangle}
\newcommand{\tl}[1]{{\tilde #1}}
\newcommand{\mysec}[1]{\vspace{5mm}\addtocounter{section}{1}{\bf #1.}\setcounter{equation}{0}}

\newtheorem{theorem}{Theorem}

\newtheorem{lemma}{Lemma}

\numberwithin{equation}{section}

\begin{document}

\title{Symmetry in systems with gyroscopic forces\footnote{Submitted on September 21, 1981.}}

\author{M.P.\,Kharlamov\footnote{Donetsk State University.}}

\date{}

\maketitle

\begin{center}
{\bf Published: \ \ \textit{Mekh. Tverd. Tela}, 1983, No. 15, pp. 87--93}\footnote{Russian Journal ``Mechanics of Rigid Body''.}

\vspace{3mm}

\href{http://www.ams.org/mathscinet-getitem?mr=0698564}{http://www.ams.org (Reference)}

\vspace{1mm}

\href{http://www.ics.org.ru/doc?pdf=1118\&dir=r}{http://www.ics.org.ru (Russian)}

\vspace{1mm}

\href{https://www.researchgate.net/publication/259360363}{https://www.researchgate.net (Russian)}

\end{center}

\begin{abstract}
We consider a generalization of the notion of a natural mechanical system to the case of additional forces of gyroscopic type. Such forces appear, for example, as a result of global reduction of a natural system with symmetry. We study symmetries in the systems with gyroscopic forces to find out when these systems admit a global analogue of a cyclic integral. The results are applied to the problem of the motion of a rigid body about a fixed point in potential and gyroscopic forces to find the most general form of such forces admitting the global area type integral.
\end{abstract}

A natural mechanical system with a symmetry group, i.e. a group of diffeomorphisms preserving  both the kinetic energy and the potential function of the system, admits the first integral called the momentum integral \cite{Arnold} and this integral is linear and homogeneous in velocities. Here we consider the question of existence of a similar integral in the systems of a more general origin, namely, the systems with gyroscopic forces \cite{Kh1976}.

\mysec{\S 1} Let us recall the basic definition. A mechanical system with gyroscopic forces (MSGF) is a 4-tuple
\begin{equation}\label{eq01}
  (M,m,\Pi,\vk),
\end{equation}
where $M$ is a smooth manifold (configuration space), $m$ a Riemannian metric on $M$ (a scalar product in $T_x M$ smooth with respect to $x\in M$), $\Pi$ a function on $M$ (the potential energy or, shortly, the potential of the system), $\vk$ a closed 2-form on $M$ (the form of gyroscopic forces).

We denote by $\bb{,}$ the scalar product in $m$ and define the total energy $H:TM\to \bR$ of \eqref{eq01} as
\begin{equation}\label{eq02}
  H(w)=\frac{1}{2}\bb{w,w}+\Pi(p_M(w)), \qquad w\in TM.
\end{equation}
Henceforth $p_M:TM\to M$ stands for the canonical projection.

Introduce the 1-form $\tq$ on $TM$ such that for any $Y\in T_w TM$
\begin{equation}\label{eq03}
  \tq(w)Y=\bb{w,Tp_M(Y)},
\end{equation}
and the 2-form $\sg$,
\begin{equation}\label{eq04}
  \sg=\rmd \tq +p_M^*\vk.
\end{equation}

The pair $(TM,\sg)$ is a symplectic manifold \cite{Kh1976}.

We define dynamics in \eqref{eq01} by the Hamiltonian vector field $X$ on $(TM,\sg)$ generated by the function \eqref{eq02},
\begin{equation}\label{eq05}
  \rmi _X \sg=-\rmd H.
\end{equation}
The field $X$ is a second-order equation, i.e.,
\begin{equation}\label{eq06}
  Tp_M\circ X={\rm id}_{TM}.
\end{equation}

Let
\begin{equation}\label{eq07}
  \Psi=\{\psi^\tau: M\to M, \; \tau \in \bR\}
\end{equation}
be a one-parameter group of diffeomorphisms of $M$. Its action is expanded to $TM$ by tangent maps. Define the generating vector fields
\begin{eqnarray}
  v(x) &=& (d/d\tau)|_{\tau=0}\ \psi^{\tau}(x), \qquad x \in M, \label{eq08}\\
  v_T(w) &=& (d/d\tau)|_{\tau=0}\ T\psi^{\tau}(w), \quad w \in TM. \label{eq09}
\end{eqnarray}

\vskip2mm
\textbf{Definition}. \textit{A group $\Psi$ is called a symmetry group of system \eqref{eq01} if all $\psi^{\tau}$ preserve the metric $m$, the potential $V$, and the form $\vk$.}
\vskip2mm

In particular, for a symmetry group we have
\begin{eqnarray}
   & v\Pi \equiv 0, \qquad v_TH\equiv 0, \label{eq10}\\
  & L_v\vk \equiv 0. \label{eq11}
\end{eqnarray}
Here $L_v$ denotes the Lie derivative along the field \eqref{eq08}.

For $\vk\equiv 0$, conditions \eqref{eq10} define a symmetry group of the corresponding natural mechanical system $(M,m,\Pi)$. This group generates the first integral $\tl{G}:TM\to \bR$ of the dynamical system $X$ as follows (see e.g. \cite{Ta1973, Kh1976}):
\begin{equation}\label{eq12}
  \tl{G}(w)=\bb{v(x),w}, \qquad w\in T_xM.
\end{equation}
For $\vk \ne 0$, this function, of course, will not any more be constant on trajectories of system \eqref{eq01}. Let $\frac{d}{dt}$ denote the derivative of functions along the field $X$ defined by equation \eqref{eq05}.

\begin{lemma}\label{lem1}
For the function $\tl{G}$ the following equation holds
\begin{equation*}
  \frac{d\tl{G}}{dt}=\rmi _ v \vk,
\end{equation*}
where $\rmi _ v \vk$, being $1$-form on $M$, is considered as a function on $TM$ linear on fibers.  \end{lemma}

\begin{proof}
From \eqref{eq08} and \eqref{eq09},
\begin{equation*}
  Tp_M(v_T(w))=\left.\frac{d}{d\tau}\right|_{\tau=0}p_M\circ T\psi^{\tau}(w)=\left.\frac{d}{d\tau}\right|_{\tau=0}\psi^{\tau}\circ p_M(w)=v(p_M(w)).
\end{equation*}
Hence, by definition \eqref{eq03} the function $\tl{G}$ can be written in the form
\begin{equation}\label{eq13}
  \tl{G}(w)=\bb{w,v}=\bb{w,Tp_M(v_T)}=(\rmi_{ v_T}\tq)(w).
\end{equation}
Therefore,
\begin{equation}\label{eq14}
  \frac{d\tl{G}}{dt}=X\,\rmi_{ v_T}\tq = \rmd (\rmi_{ v_T}\tq)(X).
\end{equation}

The form $\tq$ is $\Psi$-invariant \cite{Kh1976}, $L_{v_T}\tq=0$. Using the known identity for Lie derivatives, inner products and exterior derivatives
\begin{equation}\label{eq15}
  L_Y=\rmd\,\rmi _Y+\rmi_Y \rmd,
\end{equation}
we find
\begin{equation}\label{eq16}
  \rmd \, (\rmi_{ v_T}\tq) = - \rmi_{ v_T} \rmd \tq.
\end{equation}
From \eqref{eq14}, \eqref{eq15}, in virtue of \eqref{eq04}, \eqref{eq05}, and \eqref{eq10},
\begin{equation*}
  \frac{d\tl{G}}{dt}=- (\rmi_{ v_T} \rmd \tq)(X)=(\rmi_X\,\rmd \tq)(v_T)=-(\rmd H+\rmi_X p_M^*\vk)(v_T)=-v_T H+\vk(v,Tp_M(X)).
\end{equation*}
Thus according to \eqref{eq06} for all $w\in T_xM$ we have
\begin{equation*}
  \frac{d\tl{G}}{dt}(w)=\vk(v(x),w),
\end{equation*}
and this is the statement.
\end{proof}

Note that locally, in the regions where $\vk$ is exact, the motions of MSGF \eqref{eq01} are described by the Lagrange equations with a Lagrange function containing terms linear in velocities. The cyclic integral of such equations is linear in velocities, but not homogeneous. This is the reason to search for an integral corresponding to the group \eqref{eq07} in the form
\begin{equation}\label{eq17}
  G(w)=\bb{v,w}+f\circ p_M,
\end{equation}
where $f:M\to \bR$ is some function.

\begin{theorem}\label{theo1}
Let a one-parameter group act on the configuration space of a mechanical system with gyroscopic forces preserving the Riemannian metric and the potential. Let $v$ be the generating field on $M$ and $\vk$ the form of gyroscopic forces. Then an integral of the type \eqref{eq17} exists if and only if the $1$-form $\rmi_v\vk$ is exact on $M$. In this case, function $f$ in \eqref{eq17} is defined by the equation
\begin{equation}\label{eq18}
  \rmi_v\vk =-\rmd f.
\end{equation}
\end{theorem}
\begin{proof}
Again, we use the property \eqref{eq06}. For any function $f$ on $M$ and any $w\in TM$ we have
\begin{equation*}
\rmd f(w)=\rmd f (Tp_M\circ X(w))=\rmd (f\circ p_M)(X(w))=\frac{d}{dt}(f\circ p_M)(w),
\end{equation*}
so
\begin{equation}\label{eq19}
  \rmd f=\frac{d}{dt}(f\circ p_M).
\end{equation}

Suppose \eqref{eq18} holds. By Lemma~\ref{lem1} and \eqref{eq19}
\begin{equation*}
  \frac{d\tl{G}}{dt}=\rmi_{v}\vk=-\rmd f=-\frac{d}{dt}(f\circ p_M),
\end{equation*}
then
\begin{equation*}
  \frac{dG}{dt}=\frac{d\tl{G}}{dt}+\frac{d}{dt}(f\circ p_M)=0.
\end{equation*}

Conversely, let $G=\tl{G}+f\circ p_M$ be the first integral. Then
\begin{equation*}
  0=\frac{dG}{dt}=\frac{d\tl{G}}{dt}+\frac{d}{dt}(f\circ p_M)=\rmi_v\vk+\rmd f
\end{equation*}
yields \eqref{eq18}, so $\rmi_v\vk$ is exact.
\end{proof}

\textbf{Remarks}. $1^\circ$. If $\rmi_v\vk$ is exact, then automatically the form $\vk$ of gyroscopic forces is $\Psi$-invariant. Indeed, $\vk$ is closed, so from \eqref{eq15} we have
\begin{equation}\label{eq20}
  L_v\vk=\rmd\,\rmi_v \vk +\rmi_v \rmd \vk =0.
\end{equation}
The converse statement is not true. If \eqref{eq11} holds, then $\rmi_v \vk$ only has to be closed. Therefore, the set of local integrals of the type \eqref{eq17} exist, but in general case it is not possible to glue them into one global integral. In the case when the cohomology group $H^1(M,\bR)$ is trivial the global existence of an integral of the type \eqref{eq17} is equivalent to $\Psi$-invariance of the form of gyroscopic forces.

$2^\circ$. The exactness of the form $\rmi_v\vk$ does not imply the exactness of $\vk$. Therefore, an integral of the type \eqref{eq17} can exist even in the case when the system does not admit a global Lagrangian.

$3^\circ$. Suppose both $\rmi_v\vk$ and $\vk$ are exact. Does it mean that there exists a 1-form $\lambda$ such that $\rmd \lambda=\vk$ and $\lambda$ is $\Psi$-invariant? If not, then there exist globally Lagrangian systems having a global integral of the type \eqref{eq17} but not having a global $\Psi$-invariant Lagrangian (note that locally the answer is always positive). In this case Theorem~\ref{theo1} could be a generalization of Noether's theorem \cite{Arnold} for Lagrangian systems.

\mysec{\S 2} As an example, let us consider the problem of the motion of a rigid body about a fixed point. To any position of the body we, as usual, assign the matrix
\begin{equation}\label{eq21}
Q = \begin{Vmatrix} \aq_1 & \aq_2 & \aq_3\\
                    \aq'_1 & \aq'_2 & \aq'_3 \\
                    \aq''_1 & \aq''_2 & \aq''_3
                           \end{Vmatrix}
\end{equation}
the columns of which are the components of the unit vectors of the moving frame in an orthonormal frame fixed in inertial space.

The group $SO(3)$ of matrices \eqref{eq21} is identified with the submanifold in $\bR^9(\aq_1,\ldots,\aq''_3)$ defined by the relations
\begin{equation}\label{eq22}
  \aq_i\aq_j+\aq'_i\aq'_j+\aq''_i\aq''_j=\delta_{ij}, \qquad i,j=1,2,3
\end{equation}
or, equivalently,
\begin{equation}\label{eq23}
  \aq^{(i)}_1\aq^{(j)}_1+\aq^{(i)}_2\aq^{(j)}_2+\aq^{(i)}_3\aq^{(j)}_3=\delta_{ij}, \qquad i,j=0,1,2.
\end{equation}
Here $\delta_{ij}$ is the Kronecker delta.

Consider the following vector fields on $SO(3)$
\begin{equation}\label{eq24}
  \Oq_i: SO(3) \to TSO(3), \qquad i=1,2,3
\end{equation}
defined via the basic fields on $\bR^9$ as
\begin{equation}\label{eq25}
\begin{array}{l}
\ds  \Oq_1(Q)=\aq_3\frac{\prt}{\prt \aq_2}-\aq_2\frac{\prt}{\prt \aq_3}+\aq'_3\frac{\prt}{\prt \aq'_2}-\aq'_2\frac{\prt}{\prt \aq'_3}+\aq''_3\frac{\prt}{\prt \aq''_2}-\aq''_2\frac{\prt}{\prt \aq''_3},\\
\ds  \Oq_2(Q)=\aq_1\frac{\prt}{\prt \aq_3}-\aq_3\frac{\prt}{\prt \aq_1}+\aq'_1\frac{\prt}{\prt \aq'_3}-\aq'_3\frac{\prt}{\prt \aq'_1}+\aq''_1\frac{\prt}{\prt \aq''_3}-\aq''_3\frac{\prt}{\prt \aq''_1},\\
\ds  \Oq_3(Q)=\aq_2\frac{\prt}{\prt \aq_1}-\aq_1\frac{\prt}{\prt \aq_2}+\aq'_2\frac{\prt}{\prt \aq'_1}-\aq'_1\frac{\prt}{\prt \aq'_2}+\aq''_2\frac{\prt}{\prt \aq''_1}-\aq''_1\frac{\prt}{\prt \aq''_2}.
\end{array}
\end{equation}
The corresponding one-parameter groups rotate the body about the axes fixed with respect to the body and defined by the moving frame. The basis \eqref{eq24} in $TSO(3)$ is non-holonomic
\begin{equation}\label{eq26}
  [\Oq_2,\Oq_3]=\Oq_1, \qquad   [\Oq_3,\Oq_1]=\Oq_2, \qquad   [\Oq_1,\Oq_2]=\Oq_3.
\end{equation}

According to \eqref{eq22} and \eqref{eq23}, the 1-forms $\Lqq_i$ ($i=1,2,3$) on $SO(3)$ are well-defined by the following equations
\begin{equation}\label{eq27}
\begin{array}{l}
  \Lqq_1(Q)=\aq_3 \rmd\aq_2+\aq'_3 \rmd\aq'_2+\aq''_3 \rmd\aq''_2=-(\aq_2 \rmd\aq_3+\aq'_2 \rmd\aq'_3+\aq''_2 \rmd\aq''_3),\\
  \Lqq_2(Q)=\aq_1 \rmd\aq_3+\aq'_1 \rmd\aq'_3+\aq''_1 \rmd\aq''_3=-(\aq_3 \rmd\aq_1+\aq'_3 \rmd\aq'_1+\aq''_3 \rmd\aq''_1),\\
  \Lqq_3(Q)=\aq_2 \rmd\aq_1+\aq'_2 \rmd\aq'_1+\aq''_2 \rmd\aq''_1=-(\aq_1 \rmd\aq_2+\aq'_1 \rmd\aq'_2+\aq''_1 \rmd\aq''_2).
\end{array}
\end{equation}
The straightforward check shows that these forms constitute the basis in $T^*TSO(3)$ dual to \eqref{eq25}. Note one inversion of relations \eqref{eq27} useful for the future, 
\begin{equation}\label{eq28}
\begin{array}{c}
  \rmd \aq_1^{(i)} = \aq_2^{(i)} \Lqq_3- \aq_3^{(i)}\Lqq_2, \quad
  \rmd \aq_2^{(i)} = \aq_3^{(i)} \Lqq_1- \aq_1^{(i)}\Lqq_3, \quad
  \rmd \aq_3^{(i)} = \aq_1^{(i)} \Lqq_2- \aq_2^{(i)}\Lqq_1 \\
  (i=0,1,2).
\end{array}
\end{equation}

To each state of the body
\begin{equation}\label{eq29}
  \Oq_Q\in T_Q SO(3)=Q\;{\rm Ass}(3),
\end{equation}
using the canonical isomorphism $\mbf: {\rm Ass}(3) \to \bR^3$ \cite{Kh1976}, we assign the vector of the inner angular velocity
\begin{equation}\label{eq30}
  \bo =\mbf(Q^{-1}\Oq_Q).
\end{equation}
The components of this vector are the projections of the real angular velocity of the state $\Oq_Q$ to the moving axes. The corresponding map
\begin{equation}\label{eq31}
{\rm tr}:TSO(3)\to SO(3){\times}\bR^3, \qquad {\rm tr}(\Oq_Q)=(Q,\bo)
\end{equation}
is a known trivialization of $TSO(3)$ \cite{Arnold,Ta1973}.

Properties \eqref{eq26} and the duality of \eqref{eq25} and \eqref{eq27} yield the following.

\begin{lemma}\label{lem2}
The form $\Lqq_i$ assigns to infinitesimal rotation $\Oq_Q$ the $i^{\rm th}$ component of its inner angular velocity. Forms \eqref{eq27} are left invariant and satisfy
\begin{equation}\label{eq32}
  \rmd \Lqq_1=\Lqq_3 \wedge \Lqq_2, \qquad   \rmd \Lqq_2=\Lqq_1 \wedge \Lqq_3, \qquad  \rmd \Lqq_3=\Lqq_2 \wedge \Lqq_1.
\end{equation}
\end{lemma}

Consider the second tangent bundle $TTSO(3)$. The map tangent to \eqref{eq31} takes each fiber $T_{\Oq_Q}TSO(3)$ to the product $T_Q SO(3){\times}T_\bo \bR^3$. The last factor is naturally identified with $\bR^3$. Each element $X\in TTSO(3)$ then has the form
\begin{equation}\label{eq33}
  X=((Q,\bo),(\dot Q, \bde)),
\end{equation}
where the symbol $\dot Q$ stands for any matrix satisfying $Q^{-1}\dot Q \in {\rm Ass}(3)$, and $\bde \in \bR^3$. We have
\begin{equation}\label{eq34}
\begin{array}{c}
  p_{TSO(3)}((Q,\bo),(\dot Q, \bde))=(Q,\bo), \\
  Tp_{SO(3)}((Q,\bo),(\dot Q, \bde))=(Q,\mbf(Q^{-1}\dot Q)).
\end{array}
\end{equation}

Let us lift the fields \eqref{eq25} to $TSO(3)$ by putting $\bde=0$. In $T_{\Oq_Q}TSO(3)$, we obtain the following basis
\begin{equation}\label{eq35}
  \Oq_1(Q),\;\Oq_2(Q),\;\Oq_3(Q),\;\frac{\prt}{\prt\oq_1},\; \frac{\prt}{\prt\oq_2},\; \frac{\prt}{\prt\oq_3}.
\end{equation}

Analogously, for the forms $\Lqq_i$ lifted to $T^*TSO(3)$ by the pull-back $p^*_{SO(3)}$ we keep the same notation, i.e., for $X\in T_{\Oq_Q}SO(3)$ we by definition put
\begin{equation}\label{eq36}
  \Lqq_i(Q,\bo)(X)=\Lqq_i(Q)(Tp_{SO(3)}(X)).
\end{equation}
Therefore the basis in $T^*_{\Oq_Q}TSO(3)$ dual to \eqref{eq35} is
\begin{equation}\label{eq37}
  \Lqq_1(Q), \;\Lqq_2(Q), \;\Lqq_3(Q), \; \rmd \oq_1, \;\rmd \oq_2, \;\rmd \oq_3.
\end{equation}

Using trivialization \eqref{eq31}, let us define a metric $m$ on $SO(3)$ by
\begin{equation}\label{eq38}
  m_Q(\Oq_Q^1,\Oq_Q^2)=\mbA\bo^1\cdot\bo^2,
\end{equation}
where $\mbA:\bR^3\to \bR^3$ is a symmetric operator (inertia tensor) and the central dot symbol denotes the standard scalar product in $\bR^3$.

Let $\Pi=\Pi(Q)$ be a smooth function on $SO(3)$ and
\begin{equation}\label{eq39}
  \vk =\vk_1\Lqq_1+\vk_2\Lqq_2+\vk_3\Lqq_3
\end{equation}
a closed 2-form on $SO(3)$. The representation \eqref{eq39} with $\vk_i: SO(3)\to \bR$ is possible due to \eqref{eq32}. The condition $\rmd \vk=0$ takes the compact form
\begin{equation}\label{eq40}
  \Oq_1 \vk_1+\Oq_2 \vk_2+\Oq_3 \vk_3\equiv 0.
\end{equation}

Finally, the problem of the motion of a rigid body about a fixed point is formalized as a mechanical system with gyroscopic forces
\begin{equation}\label{eq41}
  (SO(3),m,\Pi,\vk).
\end{equation}

Let us obtain the representation of the Lagrange forms \eqref{eq03} and \eqref{eq04} in the basis \eqref{eq37}. From \eqref{eq33}, \eqref{eq34}, and \eqref{eq38},
\begin{equation*}
  \tq(\Oq_Q)(X)=m_Q(\Oq_Q,\dot Q)=\mbA \bo\cdot \mbf(Q^{-1}\dot Q).
\end{equation*}
By Lemma~\ref{lem2} this yields
\begin{equation*}
  \tq(Q,\bo)=\sum_{i,j=1}^3 A_{ij}\oq_i\Lqq_j(Q)
\end{equation*}
($A_{ij}$ are the components of $\mbA$). Applying the exterior derivative, we get
\begin{equation}\label{eq42}
\rmd  \tq(Q,\bo)=\sum_{i,j=1}^3 A_{ij}( \rmd \oq_i\wedge \Lqq_j(Q)+ \oq_i\wedge \rmd \Lqq_j(Q)).
\end{equation}

Let
\begin{equation}\label{eq43}
  X(Q,\bo)=((Q,\bo),(\dot Q(Q,\bo), \bde(Q,\bo)))
\end{equation}
be the vector field defining dynamics of system \eqref{eq41}. Then
\begin{equation}\label{eq44}
  \frac{dQ}{dt}=\dot Q, \qquad \frac{d\bo}{dt}=\bde.
\end{equation}
Since $X$ is a second-order equation, the second equality \eqref{eq34} yields
\begin{equation}\label{eq45}
  \bo = \mbf(Q^{-1}\dot Q).
\end{equation}
Thus $\dot Q=Q\mbf^{-1}(\bo)$. Due to \eqref{eq21} and \eqref{eq44} we have, for the columns of the matrix $Q^T=Q^{-1}$,
\begin{equation}\label{eq46}
  \frac{d{\boldsymbol \aq}}{dt}={\boldsymbol \aq}\times \bo,\qquad
  \frac{d{\boldsymbol \aq}'}{dt}={\boldsymbol \aq}'\times \bo,\qquad
  \frac{d{\boldsymbol \aq}''}{dt}={\boldsymbol \aq}''\times \bo.
\end{equation}
These are the Poisson equations.

Let us find the terms in the general equation given by \eqref{eq05}. According to \eqref{eq02} and \eqref{eq38},
\begin{equation}\label{eq47}
  H(Q,\bo)=\frac{1}{2}\mbA \bo \cdot \bo +\Pi(Q).
\end{equation}

Let us suppose for simplicity that the moving frame is chosen to make the tensor $\mbA$ diagonal. Then from \eqref{eq28} we have
\begin{equation}\label{eq48}
\begin{array}{l}
\ds    \rmd H=A_1\oq_1\rmd \oq_1+A_2\oq_2\rmd \oq_2+A_3\oq_3\rmd \oq_3+ \sum_{i=0}^2 \left[\frac{\prt\Pi}{\prt\aq_1^{(i)}}(\aq_2^{(i)} \Lqq_3- \aq_3^{(i)}\Lqq_2) + \right. \\ 
\ds \qquad \left. +\frac{\prt\Pi}{\prt\aq_2^{(i)}}(\aq_3^{(i)} \Lqq_1- \aq_1^{(i)}\Lqq_3)+\frac{\prt\Pi}{\prt\aq_3^{(i)}}(\aq_1^{(i)} \Lqq_2- \aq_2^{(i)}\Lqq_1)\right].
\end{array}
\end{equation}
Note that by definition
\begin{equation*}
  \rmd \oq_i(Q,\bo)(X)=X\oq_i=\frac{d\oq_i}{dt},
\end{equation*}
and according to \eqref{eq45} and \eqref{eq36}
\begin{equation*}
  \Lqq_i(Q,\bo)(X)=\oq_i.
\end{equation*}
Therefore from \eqref{eq39} and \eqref{eq42} we find
\begin{equation}\label{eq49}
\begin{array}{rcl}
  \rmi _X \rmd \tq & = & \ds \Bigl[A_1\frac{d\oq_1}{dt}+(A_3-A_2)\oq_2\oq_3\Bigr]\Lqq_1+ \Bigl[A_2\frac{d\oq_2}{dt}+(A_1-A_3)\oq_3\oq_1\Bigr]\Lqq_2+\\[3mm]
  {}& + & \ds \Bigl[A_3\frac{d\oq_3}{dt}+(A_2-A_1)\oq_1\oq_2\Bigr]\Lqq_3-\\[3mm]
  {} & - & (A_1\oq_1\rmd \oq_1+A_2\oq_2\rmd \oq_2+A_3\oq_3\rmd \oq_3),\\[3mm]
  \rmi_X p^*_ {SO(3)}\vk & = & (\vk_3\oq_2-\vk_2\oq_3)\Lqq_1+(\vk_1\oq_3-\vk_3\oq_1)\Lqq_2+(\vk_1\oq_3-\vk_3\oq_1)\Lqq_3.
\end{array}
\end{equation}

Let us substitute \eqref{eq48} and \eqref{eq49} into \eqref{eq04} and \eqref{eq05}. Identifying the coefficients of the independent forms $\Lqq_i$, we come to the Euler equations
\begin{equation}\label{eq50}
\begin{array}{l}
  \ds A_1\frac{d\oq_1}{dt}+(A_3-A_2)\oq_2\oq_3+\vk_3\oq_2-\vk_2\oq_3=\sum_{i=0}^2\Bigl(\aq_2^{(i)}\frac{\prt \Pi}{\prt \aq_3^{(i)}}-\aq_3^{(i)}\frac{\prt \Pi}{\prt \aq_2^{(i)}}\Bigr),\\
  \ds A_2\frac{d\oq_2}{dt}+(A_1-A_3)\oq_3\oq_1+\vk_1\oq_3-\vk_3\oq_1=\sum_{i=0}^2\Bigl(\aq_3^{(i)}\frac{\prt \Pi}{\prt \aq_1^{(i)}}-\aq_1^{(i)}\frac{\prt \Pi}{\prt \aq_3^{(i)}}\Bigr),\\
  \ds A_3\frac{d\oq_3}{dt}+(A_2-A_1)\oq_1\oq_2+\vk_2\oq_1-\vk_1\oq_2=\sum_{i=0}^2\Bigl(\aq_1^{(i)}\frac{\prt \Pi}{\prt \aq_2^{(i)}}-\aq_2^{(i)}\frac{\prt \Pi}{\prt \aq_1^{(i)}}\Bigr).
\end{array}  
\end{equation}
This system is closed by equations \eqref{eq46}. We emphasize that $\vk_i$ are arbitrary functions on $SO(3)$ with the only condition \eqref{eq40}. Thus we obtained the most general equations of the type \eqref{eq05} describing the motion of a rigid body in potential and gyroscopic force fields.

\mysec{\S 3} Suppose system \eqref{eq41} has a symmetry group $\Psi$ which rotates the body around the first axis fixed in space. The generating field for such group is
\begin{equation}\label{eq51}
  v(Q)=Q \mbf^{-1}({\boldsymbol \aq}),
\end{equation}
and the quotient map $p:Q\mapsto {\boldsymbol \aq}$ takes $SO(3)$ to the Poisson sphere 
\begin{equation*}
  \aq_1^2+\aq_2^2+\aq_3^2=1.
\end{equation*}
So if $F=F(Q)$ is a $\Psi$-invariant function on $SO(3)$, then in the variables $\aq_i^{(j)}$ (with $i=1,2,3$ and $j=0,1,2$) it can be written in the form
\begin{equation}\label{eq52}
  F=F(\aq_1,\aq_2,\aq_3).
\end{equation}
In particular this is the condition for the potential $\Pi$:
\begin{equation}\label{eq53}
    \Pi=\Pi(\aq_1,\aq_2,\aq_3).
\end{equation}

By Theorem~\ref{theo1}, the considered system has an integral of the type \eqref{eq17} (it can be naturally called the \textbf{area integral}) if and only if the field \eqref{eq51} and the 1-form \eqref{eq39} satisfy \eqref{eq18} with some function $f$ on $SO(3)$. In fact, \eqref{eq18} implies $f$ is $\Psi$-invariant. Indeed, $vf=\rmd f(v)=-\rmi_v\rmi_v\vk\equiv 0$. Therefore,
\begin{equation}\label{eq54}
    f=f(\aq_1,\aq_2,\aq_3).
\end{equation}
Using Lemma~\ref{lem2}, expressions \eqref{eq28}, and equations \eqref{eq51} and \eqref{eq54}, we find
\begin{equation}\label{eq55}
\begin{array}{rcl}
  \rmi_v \vk & = & (\vk_3\aq_2-\vk_2\aq_3)\Lqq_1+(\vk_1\aq_3-\vk_3\aq_1)\Lqq_2+(\vk_2\aq_3-\vk_3\aq_2)\Lqq_3, \\[2mm]
\ds  -\rmd f & = & \ds \Bigl(\frac{\prt f}{\prt \aq_3}\aq_2-\frac{\prt f}{\prt \aq_2}\aq_3   \Bigr)\Lqq_1+ \Bigl(\frac{\prt f}{\prt \aq_1}\aq_3-\frac{\prt f}{\prt \aq_3}\aq_1   \Bigr)\Lqq_2+\\[3mm]
{}&+ & \ds  \Bigl(\frac{\prt f}{\prt \aq_2}\aq_1-\frac{\prt f}{\prt \aq_1}\aq_2   \Bigr)\Lqq_3.
\end{array}
\end{equation}

Denote
\begin{equation}\label{eq56}
  {\boldsymbol \kappa}=\left( \begin{array}{c}\vk_1\\ \vk_2\\ \vk_3\end{array}\right), \quad   \grad f=\left( \begin{array}{c}\prt f / \prt \aq_1\\ \prt f / \prt \aq_2\\ \prt f / \prt \aq_3\end{array}\right) \in \bR^3.
\end{equation}

Substitute \eqref{eq55} and \eqref{eq56} in \eqref{eq18} to obtain
\begin{equation}\label{eq57}
  ({\boldsymbol \kappa}-\grad f)\times {\boldsymbol \aq}\equiv 0.
\end{equation}

Using the fact that the forms $\Lqq_i$ are left invariant and, in particular, $L_v \Lqq_i \equiv 0$, we get from \eqref{eq39},
\begin{equation}\label{eq58}
  L_v \vk=(v \vk_1)\Lqq_1+(v \vk_2)\Lqq_2+(v \vk_3)\Lqq_3.
\end{equation}
Comparing \eqref{eq20} with \eqref{eq58} we see that $v\vk_i\equiv 0$, i.e.,
\begin{equation}\label{eq59}
      \vk_i =\vk_i (\aq_1,\aq_2,\aq_3), \qquad i=1,2,3.
\end{equation}
From \eqref{eq57} and \eqref{eq59} we get the following result.
\begin{theorem}\label{theo2}
In system \eqref{eq41} with the potential of the type \eqref{eq53}, the area integral exists  if and only if the vector composed by the coefficients of the form of gyroscopic forces has a representation
\begin{equation}\label{eq60}
{\boldsymbol \kappa}=F {\boldsymbol \aq}+\grad f,
\end{equation}
where $F$ and $f$ are functions of the type \eqref{eq52} and \eqref{eq54}.
\end{theorem}

If the area integral exists, then from \eqref{eq17}, \eqref{eq38}, and \eqref{eq51} we find its form
\begin{equation}\label{eq61}
  G(Q,\bo)=A_1\oq_1 \aq_1+A_2\oq_2 \aq_2+A_3\oq_3 \aq_3+f(\aq_1,\aq_2,\aq_3).
\end{equation}

In the partial case when $\vk_i \equiv {\rm const}$ $(i=1,2,3)$ let us put $f=\vk_1\aq_1+\vk_2\aq_2+\vk_3\aq_3$ and ${F=0}$. The area integral takes the classical form
\begin{equation*}
   G=(A_1\aq_1+\vk_1)\oq_1+(A_2\aq_2+\vk_2)\oq_2+(A_3\aq_3+\vk_3)\oq_3
\end{equation*}
of the area integral for the motion of a gyrostat in a field with potential of the type \eqref{eq53}.

Let us emphasize the following simple, but important fact. Let the condition \eqref{eq60} hold. Then the vector field \eqref{eq43}, being a Hamiltonian system on the symplectic manifold $(TSO(3),\sg)$, has two first integrals \eqref{eq47} and \eqref{eq61} and, obviously, their Poisson bracket is zero: $\{H,G\}=XG\equiv 0$. By Liouville's theorem \cite{Arnold}, to solve this problem in quadratures, it is sufficient to point out one more integral independent of $H$ and $G$ and this integral should be in involution with $G$. The following criterion holds.
\begin{theorem}\label{theo3}
A function $K:TSO(3) \to \bR$ is in involution with the area integral if and only if $K$ is preserved by the maps tangent to diffeomorphisms of the group $\Psi$.  
\end{theorem}

This statement is true for the general case. From \eqref{eq12}, \eqref{eq13}, and \eqref{eq16}\,--\,\eqref{eq18} we have 
\begin{equation*}
  \rmd G= -\rmi_{v_T}\rmd \tq - p_M^*\rmi_v \vk=-\rmi_{v_T}\sg,
\end{equation*}
i.e., the Hamiltonian field generated by the Hamilton function $G$ is $v_T$. Then from the definition of the Poisson bracket we have $\{K,G\}=v_T K$. This proves the theorem.

\vskip2mm

Returning to the rigid body problem we obtain that the needed additional integral must depend only on the variables $\oq_1,\oq_2,\oq_3,\aq_1,\aq_2,\aq_3$.

\end{document}